\documentclass[11pt]{amsart}
\usepackage{mathrsfs}
\usepackage{amssymb}
\usepackage{comment}
\usepackage{graphicx}
\usepackage{hyperref}
\usepackage{palatino}
\usepackage{tikz-cd}

\numberwithin{equation}{section}
\theoremstyle{plain}

\newtheorem{Corollary}[equation]{Corollary}
\newtheorem*{Corollary*}{Corollary}
\newtheorem{Theorem}[equation]{Theorem}
\newtheorem*{Theorem*}{Theorem}
\newtheorem{Lemma}[equation]{Lemma}
\theoremstyle{definition}

\newtheorem{Example}[equation]{Example}
\newtheorem{Remark}[equation]{Remark}

\usepackage{mathtools}
\mathtoolsset{showonlyrefs}

\allowdisplaybreaks

\usepackage{array}

\newcommand\undermat[2]{%
	\makebox[0pt][l]{$\smash{\underbrace{\phantom{%
					\begin{matrix}#2\end{matrix}}}_{\text{$#1$}}}$}#2}

\def\C{\mathbb{C}}

\def\D{\mathbb{D}}

\def\N{\mathbb{N}}

\def\phi{\varphi}

\newcommand{\beqa}{\begin{eqnarray*}}
\newcommand{\eeqa}{\end{eqnarray*}}
\newcommand{\dst}{\displaystyle}

\newcommand{\eps}{\varepsilon}

\renewcommand{\le}{\leqslant}
\renewcommand{\leq}{\leqslant}
\renewcommand{\ge}{\geqslant}
\renewcommand{\geq}{\geqslant}
\renewcommand{\subset}{\subseteq}

\title[Estimating the solutions of B\'ezout's   identity]{An analytic approach to estimating the solutions of  B\'ezout's polynomial  identity}

\author[Fricain]{Emmanuel Fricain}
 \address{Laboratoire Paul Painlev\'e, Universit\'e de Lille, 59 655 Villeneuve d'Ascq C\'edex }
 \email{emmanuel.fricain@univ-lille.fr}

\author[Hartmann]{Andreas Hartmann}
\address{Univ. Bordeaux, CNRS, Bordeaux INP, IMB, UMR 5251, F-33400, Talence, France}
\email{Andreas.Hartmann@math.u-bordeaux.fr}

\author[Ross]{William T. Ross}
	\address{Department of Mathematics and Statistics, University of Richmond, Richmond, VA 23173, USA}
	\email{wross@richmond.edu}
	
		\author[Timotin]{Dan Timotin}
	\address{Simion Stoilow Institute of Mathematics of the Romanian Academy, Calea Grivi\c tei 21, Bucharest 010702, Romania}
	\email{Dan.Timotin@imar.ro}
 
\begin{document}

\begin{abstract}
This paper contains sharp bounds on the coefficients of the polynomials $R$ and $S$ which solve the classical one variable B\'{e}zout identity $A R + B S = 1$, where $A$ and $B$ are polynomials with no common zeros. 
The bounds are expressed in terms of the separation of the zeros of $A$ and $B$.
Our proof involves contour integral representations of these coefficients. We also obtain an estimate on the norm of the inverse of the Sylvester matrix.
\end{abstract}
\maketitle

\section{Introduction}

The well known B\'{e}zout polynomial identity \cite{Bez} says that if $A(z)$ and $B(z)$ are complex polynomials with no common zeros, of degrees $N$ and $K$ respectively, then there are  complex  polynomials $ R(z) $ and $ S(z) $ with 
\begin{equation}\label{defgonx}
 \deg R\leq K-1 \; \; \mbox{and} \; \; \deg S\le N-1
 \end{equation} such that 
\begin{equation}\label{Bezout_C}
	A(z) R(z) + B(z) S(z) = 1 \; \mbox{for all $z \in \C$.}
\end{equation}
Moreover, if $ A , B$ are not both constant, then $R$ and $S $ are uniquely determined by the degree condition from \eqref{defgonx} and are called the \emph{minimal solutions} of~\eqref{Bezout_C}.
To avoid trivialities, we will assume for the rest of the paper that both the polynomials $A$ and $B$ are not constant.

 Our main result  (Theorem \ref{th:main} ) provides sharp (upper) estimates of the  coefficients of $R$ and $S$ in terms of the constant 
$$\delta(A,B):=\min\{ |A(z)|+|B(z)|: A(z)=0\text{ or }B(z)=0\}.$$
Namely, we prove there is a constant $C > 0$, depending only on $N$ and $K$, such that if the moduli of the coefficients of $A$ and $B$ are bounded by 1, then the moduli of the coefficients of $R$ and $S$ are bounded by $C\delta^{-2}$.

Computing the coefficients of $R$ and $S$ traditionally involves
 algebraic methods and are not convenient for obtaining sharp estimates. Our approach uses analytical tools based on 
 contour integrals similar to those in~\cite{Yger} (see also \cite{YgerNew}).  In Theorem \ref{th:estimates on Sylvester} we apply our methods to obtain an upper estimate of the norm of the  inverse of the  Sylvester matrix. 

Quite surprisingly, we could not find any estimates of coefficients of the minimal solutions $R$ and $S$ from \eqref{Bezout_C}  in the literature. 
There are estimates of the degree and size of the solutions of the much harder analogous problem in several variables, which is related to Hilbert's Nullstellensatz (see, for instance,~\cite{BY} and the references therein). However, they do not involve quantities similar to $ \delta(A,B)$ and thus are not connected  to Theorem \ref{th:main}.

One can regard B\'ezout's polynomial  identity as an algebraic version of the well-known corona theorem of Carleson \cite{MR141789}: if $\phi, \psi\in H^\infty$, the bounded analytic functions on the open unit disk $\D$, satisfy 
\begin{equation}\label{cosodgsertggdsa} 
	1\ge
|\phi(z)|+|\psi(z)|\ge \delta>0 \; \mbox{for all $z\in\D$},
\end{equation} then there are $g,h\in H^\infty$ such that 
\begin{equation}\label{zppzpZPPZPPZ}
\phi (z) g(z)+\psi(z) h(z)  = 1\; \mbox{for all $z \in \D$}.
\end{equation} 
Moreover,  results from \cite{MR570865, MR629839,Uch} show there is a universal constant $C > 0$ such that one can choose the $g,h \in H^{\infty}$ satisfying \eqref{zppzpZPPZPPZ} so that 
\begin{equation}\label{88u88799II}
\| g \|_{\infty}, \| h \|_{\infty}\le \frac{C}{\delta^2} \log \frac{1}{\delta},
\end{equation}
where $\|g\|_{\infty} := \sup\{|g(z)|: z \in \D\}$ and $\delta$ is defined in \eqref{cosodgsertggdsa}. 
A result from \cite{MR1945294} shows that the estimate in \eqref{88u88799II} is almost sharp. 
In the case of B\'{e}zout's identity, it follows from our main theorem (Theorem \ref{th:main})  that the logarithm term is not required.

A final remark: the constants $C_1, C_2, \dots$ that will appear in the estimates in this paper {\em only} depend on $N = \operatorname{deg} A$ and $K = \operatorname{deg} B$ and {\em not} on the coefficients of the polynomials $A$ and $B$. This fact will be in force in the sequel, even if not always explicitly stated.

\section{Preliminaries and the main result}\label{se:prelim}

Let $\C[z]$ denote the polynomials in the complex variable $z$ with coefficients in $\C$. For $N \in   \{1, 2, \ldots\}$, let $\C_{N}[z]$ denote the vector space of polynomials whose degree is at most $N$. 
We need a way to measure the ``size'' of an $A \in \C[z]$.
For this we use  the maximum of the moduli of the  coefficients of $A$ in that if 
$$A(z)=a_0+a_1z+\cdots+ a_Nz^N,$$ then
\[
\| A \|:=\max_{0\le i\le N}|a_i|
\]
defines a norm on $\C[z]$.

In  order to quantify the fact that $A \in \C_{N}[z]$ and $B \in \C_{K}[z]$  share no common roots, we define
\begin{equation}\label{eq:other delta}
	\delta(A,B):=\min\{ |A(z)|+|B(z)|: A(z)=0\text{ or }B(z)=0\}.
\end{equation}
We may reformulate~\eqref{eq:other delta} by
\begin{equation}\label{eq:min values A B}
	\delta=\delta(A,B):=\min\{ |A(\beta_j)|, |B(\alpha_i)|, 1 \leq i \leq N, 1 \leq j \leq K\},
\end{equation}
where $\{\alpha_i\}_{i = 1}^{N}$ are the roots of $A$ and $\{\beta_j\}_{j = 1}^{K}$ are the roots of $B$.

Another  way  to quantitatively express  the lack of common roots for $A$ and $B$  is to consider
\begin{equation}\label{eq:tilde delta}
	\widetilde{\delta}(A,B):=\min\{|A(z)|+|B(z)|: z \in \C\},
\end{equation}
which appears, for instance, in~\cite{MR3741670}.
Being the infimum of a family of continuous functions, $\widetilde\delta$ is upper semicontinous (but not necessarily continuous, see Example \ref{example-DD}), on the compact set
\begin{equation}\label{eq:product ball}
	\mathfrak{B}:=\{(A,B)\in \C_N[z]\times\C_K[z]: \max\{\| A \|,\|  B \|\}\le 1\},
\end{equation}
whence it is also bounded on $\mathfrak{B}$.
 Though $\delta(A,B)$ is used in this paper, Corollary~\ref{co:equivalence of deltas}
shows that $\delta(A, B)$ and $\widetilde{\delta}(A, B)$ are equivalent.

One way to obtain the unique minimal solutions $R$ and $S$ to \eqref{Bezout_C}  (see Section~\ref{se:Sylvester} for the details), is to show that their coefficients  are the solutions of a linear system involving a matrix called the {\em Sylvester matrix}. More precisely, if 
$$A(z) = a_0 + a_1 z + \cdots + a_{N} z^{N} , \quad B(z) = b_0 + b_1 z + \cdots + b_{K} z^{K},$$ and 
$$ R(z)=r_0+\dots+r_{K - 1} z^{K - 1} , \quad S(z)=s_0+\dots+ s_{N - 1} z^{N - 1},$$
one can compare the  coefficients of each side of \eqref{Bezout_C}, after carrying  out the algebraic manipulations, to create a linear system of equations in the variables  $\{r_i\}_{i= 0}^{K - 1}$ and $\{s_j\}_{j = 0}^{N - 1}$. Applying Cramer's rule to this system, 
one obtains 
 that each solution $r_i$ and $s_j$ is a quotient of two determinants of order $N+K$ whose nonzero entries are either 1 or one of the coefficients of $A$ or $B$.
 Therefore,
\begin{equation}\label{eq:sylv formulas}
r_i=\frac{\widetilde r_i}{\Delta(A,B)} ,\quad s_j=\frac{\widetilde s_j}{\Delta(A,B)},
\end{equation}
where $ \widetilde{r}_i, \widetilde s_j $ and $\Delta(A,B)$ are fixed polynomials, of degree   $ N+K $ and with at most $ (N+K)! $ terms, in the  coefficients  $\{a_j\}_{j = 0}^{N}$ and $\{b_j\}_{j = 0}^{K}$ of $ A $ and $ B $ respectively. An easy consequence of these formulas is the following lemma.

\begin{Lemma}\label{le:sylvester}
	For fixed $N$ and $K$,  let $\C_{N}^{*}[z]$ and $\C^{*}_{K}[z]$ denote the nonconstant polynomials in $\C_{N}[z]$ and $\C_{K}[z]$ respectively. Define
	\begin{equation}\label{TTTTTTttt}
	\mathcal{T} = \{(A, B) \in \C_{N}^{*}[z] \times \C_{K}^{*}[z]: \Delta(A, B) \not = 0\}
	\end{equation} and let 
	$\phi:\mathcal{T}  \mapsto \C_{K - 1}[z] \times \C_{N - 1}[z]$ be defined by $\phi(A, B) = (R, S)$, where $R$ and $S$ are the minimal solutions to \eqref{Bezout_C}. Then $\phi$ is a continuous map. 
\end{Lemma}

The quantity $\Delta(A,B) $ in \eqref{eq:sylv formulas} is the determinant of the matrix of the system and is called the {\em Sylvester resultant}. One can show (see, for instance, \cite[IV, $\S8$]{MR1878556}) that
\begin{align}\label{eq:sylv}
|\Delta(A,B)|& =\Big|a_{N}^{K} b_{K}^{N} \prod_{i = 1}^{N} \prod_{j = 1}^{K} (\alpha_i - \beta_{j})\Big|\nonumber\\
& =|b_K|^N\prod_{j=1}^{K}|A(\beta_j)|
=|a_N|^K \prod_{i=1}^{N} |B(\alpha_i)|,
\end{align}
whence the set $\mathcal{T}$ from \eqref{TTTTTTttt} is actually the set of pairs of nonconstant polynomials without common roots.

Since $ \widetilde{r}_i, \widetilde s_j $ are fixed polynomials of degree at most $ N+K $ in the coefficients of $ A $ and $ B $, the conditions $\| A \|, \| B \|\le 1$ imply that $|\widetilde{r}_i|$ and $|\widetilde{s}_j|$ are bounded above  by some positive constant depending only on $N$ and $K$.  From~\eqref{eq:sylv formulas} and~\eqref{eq:sylv},
 there is a constant $C > 0$, depending only  on $ N $ and $ K $, such that if $ \| A \|, \| B \|\le 1 $ and $ |b_K|=|a_N|=1 $ (i.e., $A$ and $B$ are monic polynomials), then 
\begin{equation}\label{eq:crude estimate RS}
 \|R\|,\|S\|\le \frac{C}{\delta^{\min(N,K)}}.
\end{equation}

Our main result below, which will be proved in Section~\ref{se:main}, is a significant improvement of~\eqref{eq:crude estimate RS} in that one can replace the exponent $ \min(N,K) $ by~2, as well as ignore the assumption that $A$ and $B$ are monic polynomials.  

\begin{Theorem}\label{th:main}
	Let $A, B \in \C[z]$, with $\operatorname{deg} A = N$ and $\operatorname{deg} B = K$, satisfy $ \| A \|, \| B \|\le 1 $.
	If $\delta = \delta(A,B)> 0$, then the unique minimal solutions
	$R \in \C_{K - 1}[z]$ and $S \in \C_{N - 1}[z]$ to B\'{e}zout's identity \eqref{Bezout_C} satisfy 
	\begin{equation}\label{eq:estimate RS}
	\|R\|,\|S\|\le \frac{C}{\delta^2},
	\end{equation}
	for some universal constant $C > 0$  depending only on $N$ and $K$.
\end{Theorem}

In Section \ref{tennn} we address the issue as to what happens when we remove the assumption that  $ \| A \|, \| B \|\le 1 $.


As noted above, we will obtain $R$ and $S$ using complex function theory. A similar method was used to explore polynomials of several variables in connection with B\'ezout identities and Hilbert's Nullstellensatz \cite{MR1249478}.  



\section{A separation lemma}

This section contains a separation lemma that will have several interesting consequences. 
Recall that 
$$A(z)=\sum_{k=0}^{N}a_kz^k \; \;  \mbox{and} \; \;  B(z)=\sum_{k=0}^{K}b_kz^k$$  
are two complex polynomials. 
If $\{\alpha_i\}_{i = 1}^{N}$ are the roots of $A$ and $\{\beta_{j}\}_{j = 1}^{K}$ are the roots of $B$, then we can write 
\begin{equation}\label{ANBBB}
A(z) = a_{N} \prod_{i = 1}^{N} (z - \alpha_i) \; \;  \mbox{and} \; \;  B(z) = b_{K} \prod_{j = 1}^{K} (z - \beta_j).
\end{equation}
For $\epsilon > 0$ let
\begin{equation}\label{LLLLLL}
L(A, \epsilon):=\{z\in\C:|A(z)|<\epsilon\},
\end{equation}
 and similarly defined for $B$, denote a sub-level set for $A$ (respectively $B$). If $A$ and $B$ share no common zeros, then certainly $L(A, \epsilon_A) \cap L(B, \epsilon_B) = \varnothing$ for small enough $\epsilon_A, \epsilon_{B}$. The next lemma, which plays a crucial role in the proof of our main theorem, shows that one can choose $\epsilon_A, \epsilon_B$ to only depend on $\delta$, $N$, and $K$ (and not on the coefficients of $A$ and $B$).

\begin{Lemma}\label{le:separation}
	If $\|A\|, \| B \|\le 1$ and $\delta = \delta(A, B) > 0$, then 
	\begin{equation}\label{eq:separation formula}
			L\big(A, \frac{\delta}{3^N}\big)\bigcap L\big(B, \frac{\delta}{3^K}\big)=\varnothing.
	\end{equation}
 
\end{Lemma}

\begin{proof}
	Let us use the standard  notation 
	$$D(a, r) := \{z \in \C: |z - a| < r\}, \; \;  a \in \C, r > 0.$$
	We begin by  fixing a value of $1 \leq j \leq K$. If $ z\in \C$ satisfies 
	$$|z-\alpha_i|\ge \dst\tfrac{1}{3}|\beta_j-\alpha_i | \; \;  \mbox{for all $1 \leq i \leq N$},$$ then \eqref{ANBBB} says that
	\begin{align*}
	|A(z)| =|a_N|\prod_{i=1}^N|z-\alpha_i|
	 \ge \frac{1}{3^N}|a_N|\prod_{i=1}^N
	|\beta_{j}-\alpha_i|
	=\frac{1}{3^N}|A(\beta_j)|
	\ge \frac{\delta}{3^N}.
	\end{align*}
	Thus, outside the region 
	$$\bigcup_{i=1}^N D(\alpha_i,\tfrac{1}{3}|\beta_j - \alpha_i|),$$ the function $|A|$ is
	bounded from below by $\delta/3^N$.
	This being true for every $1 \leq j \leq K$, we deduce that
	\begin{equation}\label{eq:definition of E_A}
		L\big(A,\frac{\delta}{3^N}\big) 
	\subset \bigcap_{j=1}^K\bigcup_{i=1}^N\ D(\alpha_i, \tfrac{1}{3}|\beta_j - \alpha_i|).
	\end{equation}

	Denote the set on the right hand side of the previous inclusion by $E_A$ (see Figure \ref{Fig_ALL}). 
	\begin{figure}\label{figure 1}
\begin{center}
 \includegraphics[width=.5\textwidth]{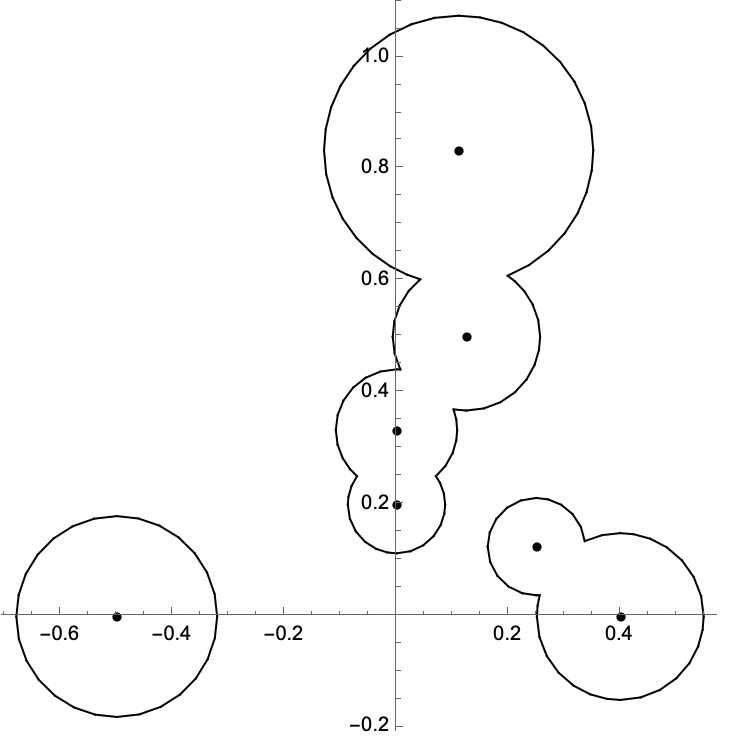}
 \caption{The regions $E_B$ (top - in one piece) and $E_A$ (bottom in two pieces) together with the zeros $\{\alpha_i\}_{i = 1}^{N}$ (in $E_A$) and the zeros $\{\beta_j\}_{j = 1}^{K}$ (in $E_{B}$). Here $A$ is the monic polynomial whose roots are $(\alpha_1, \alpha_2, \alpha_3) = (\tfrac{1}{4} + \tfrac{i}{8}, -\tfrac{1}{2}, \tfrac{2}{5})$ and $B$ is the monic polynomial whose zeros are $(\beta_1, \beta_2, \beta_3, \beta_4) = (\tfrac{1}{9} + \tfrac{5}{6}i, \tfrac{1}{8} + \tfrac{i}{2}, \tfrac{i}{3}, \tfrac{i}{5})$.}
 \label{Fig_ALL}
 \end{center}
\end{figure}
Then
	for every fixed $z\in E_A$, there is a function 
	$$\sigma:\{1, \dots, K\}\to \{1, \dots, N\}$$ such that for every 
	$1 \leq j \leq K$ we have
	\begin{equation}\label{eq:inequalities sigma}
		|z-\alpha_{\sigma(j)}|<\tfrac{1}{3}|\beta_j-\alpha_{\sigma(j)}|.
	\end{equation}

	Similarly, if 
	\begin{equation*}
	E_B:=	  \bigcap_{i=1}^N\bigcup_{j=1}^K\ D(\beta_j, \tfrac{1}{3}|\alpha_i-\beta_j|),
	\end{equation*}
	(see Figure \ref{Fig_ALL})
then $L(B,\delta/3^{K})\subset E_B$, and
	for every fixed  $z\in E_B$ there is a function 
	$$\tau:\{1, \dots, N\}\to \{1, \dots, K\}$$ such that for every $1 \leq i \leq N$ we have
	\begin{equation}\label{eq:inequalities tau}
		|z-\beta_{\tau(i)}|<\tfrac{1}{3}|\alpha_i-\beta_{\tau(i)}|.
	\end{equation}

	Towards a contradiction, suppose that
	$$L\big(A, \frac{\delta}{3^N}\big)\bigcap L\big(B, \frac{\delta}{3^K}\big) \not = \varnothing.$$ Then, using the facts that 
	$$L\big(A, \frac{\delta}{3^N}\big) \subset E_{A} \; \;  \mbox{and} \; \;  L\big(B, \frac{\delta}{3^K}\big) \subset E_{B},$$
	 it must be the case that $E_A\cap E_B\not=\varnothing$, and so there exists a $w\in E_A\cap E_B $. Let $\sigma, \tau$ denote the functions from \eqref{eq:inequalities sigma} and \eqref{eq:inequalities tau} corresponding to this $w$.
	
	Start with a $k\in\{1, \dots, K\}$ and alternately apply the functions $\sigma$ and $\tau$ to $k$, i.e., 
	$$\sigma(k),\; \tau(\sigma(k)), \;\sigma(\tau(\sigma(k))),\; \dots.$$ At some point the elements of the above sequence  must repeat. We will consider a minimal cycle in the sense that there are no repetitions. In other words, there is an integer $p$, distinct elements $j_1, \dots, j_p\in \{1, \dots, K\}$, and distinct elements $i_1, \dots, i_p\in \{1, \dots, N\}$, such that $\sigma(j_s)=i_s$ for $1\le s\le p$, while $\tau(i_s)=j_{s+1}$ for $1\le s\le p-1$, and $\tau(i_p)=j_1$ (see Figure \ref{figx}).

	For the rest of the argument, we will only use $\alpha_{\sigma(j)}$ and $\beta_{\tau(i)}$ for the selected $i$ and $j$. Therefore, to simplify things, we renumber  (corresponding to a change of notation of the indices of $\alpha$ and $\beta$) and assume that $i_s=j_s=s$ for $1\le s\le p$. With this renumbering,  the function $\sigma$ restricted to the relevant set $\{1, \dots, p\}$ becomes the identity, while $\tau(s)=s+1$ for $1\le s\le p-1$ and $\tau(p)=1$.
	
	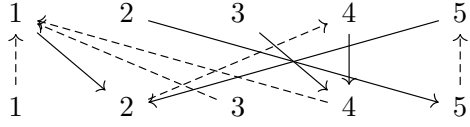
\begin{figure}

\begin{tikzcd}
	1\arrow[rd]  & 2\arrow[drrr] & 3\arrow[rd] & 4\arrow[d] & 5\arrow[dlll]\\
	1\arrow[u, dashrightarrow]  & 2\arrow[urr, dashrightarrow] & 3\arrow[ull, dashrightarrow] & 4\arrow[ulll, dashrightarrow] & 5\arrow[u, dashrightarrow]
\end{tikzcd}

\caption{An example with $N=K=5$. Solid lines represent $\sigma$ while dashed lines represent $\tau$. Starting with 3 on the top row, we obtain $4=\sigma(3),1=\tau(4), 2=\sigma(1), 4=\tau(2), 4=\sigma(4)  $. The associated minimal cycle that starts with 1 on the top row, is
$\{1, 2=\sigma(1), 4=\tau(2), 4=\sigma(4)\}$. So $p=2$ and $j_1=1, j_2=4, i_1=2, i_2=4$.}\label{figx}
		\end{figure}
	
	We may also rewrite the  inequalities~\eqref{eq:inequalities sigma} and~\eqref{eq:inequalities tau} for $z=w$ as follows, where the inequalities from \eqref{eq:inequalities sigma} appear in the left column, while those from \eqref{eq:inequalities tau} appear in the right column:
	\begin{equation}\label{eq:particular case}
		\begin{split}
			|w-\alpha_1|<\tfrac{1}{3}|\alpha_1 - \beta_1| &\qquad |w-\beta_2|<\tfrac{1}{3}| \beta_2 - \alpha_1|\\
			|w-\alpha_2|<\tfrac{1}{3}|\alpha_2 - \beta_2| &\qquad |w-\beta_3|<\tfrac{1}{3}|\beta_3 - \alpha_2|\\
			\vdots\\
			|w-\alpha_p|<\tfrac{1}{3}|\alpha_p - \beta_p| &\qquad |w-\beta_1|<\tfrac{1}{3}|\beta_{1} - \alpha_p|.
		\end{split}
	\end{equation}
	Adding the two inequalities in each line and using the triangle inequality, we  obtain a series of inequalities that do not contain the term $w$:
	
	\begin{equation*} 
		\begin{split}
			| \beta_2 - \alpha_1|&<\tfrac{1}{3}|\alpha_1 - \beta_1|+\tfrac{1}{3}|\beta_2 - \alpha_1|\\
			|\beta_3 - \alpha_2|&<\tfrac{1}{3}|\alpha_2 - \beta_2|+\tfrac{1}{3}|\beta_3 - \alpha_2|\\
			\vdots\\
			| \beta_{1} - \alpha_p|&<\tfrac{1}{3}|\alpha_p - \beta_p|+\tfrac{1}{3}|\beta_{1} - \alpha_p|
		\end{split}
	\end{equation*}
	or, equivalently, 
	\begin{equation}\label{eq:without w 1}
		\begin{split}
			|\beta_2 - \alpha_1 |&<\tfrac{1}{2}|\alpha_1 - \beta_1| \\
			|\beta_3 - \alpha_2|&<\tfrac{1}{2}|\alpha_2 - \beta_2| \\
			\vdots\\
			|\beta_{1} - \alpha_p|&<\tfrac{1}{2}|\alpha_p - \beta_p|.
		\end{split}
	\end{equation}
	Returning to~\eqref{eq:particular case}, we now add each inequality in the right column with the one in the next row of the left column (with the obvious change in the last row). Applying the triangle inequality and reducing the similar terms, we obtain 
	\begin{equation}\label{eq:without w 2}
		\begin{split}
			|\beta_2 - \alpha_2|&<\tfrac{1}{2}|\beta_2 - \alpha_1|\\
			|\beta_3 - \alpha_3|&<\tfrac{1}{2}|\beta_3 - \alpha_2|\\
			\vdots\\
			|\beta_1 - \alpha_1|&<\tfrac{1}{2}|\beta_1 - \alpha_p|.
		\end{split}
	\end{equation}
	Now add all the inequalities from \eqref{eq:without w 1} and~\eqref{eq:without w 2} to obtain 
	\[
	\begin{split}
		&\sum_{i=2}^{p}|\beta_i - \alpha_{i-1} | + |\beta_1 - \alpha_p| +  \sum_{i=1}^{p}|\beta_i - \alpha_i|
		\\
		&\qquad\qquad<\tfrac{1}{2} \left(\sum_{i=1}^{p}|\alpha_i - \beta_i|+
		\sum_{i=2}^{p}|\beta_i - \alpha_{i-1} |+|\beta_1 -\alpha_p |\right),
	\end{split}
	\]
	which is an obvious contradiction. This shows that $E_A\cap E_B=\varnothing$ which, as argued earlier, proves \eqref{le:separation}.
\end{proof}


\begin{Remark}\label{re:form of E_A} 
	Observe that  $\{\alpha_i\}_{i = 1}^{N}$ (the zeros of $A$) belong to the regions $$ \bigcup_{i=1}^N D(\alpha_i, \tfrac{1}{3}|\beta_j - \alpha_i|) \; \; \mbox{ for all $1 \leq j \leq K$},$$ and thus belong to   $ E_A $. On the other hand, for a fixed $ j $, the point $ \beta_j $ (a zero of $B$) is at a positive distance from the region $$ \bigcup_{i=1}^N D(\alpha_i,  \tfrac{1}{3}|\beta_j - \alpha_i|),$$ and therefore at a positive distance from $ E_A $.  
	Furthermore, any $z\in\partial E_A$ must satisfy
 $|A(z)|\ge \delta/3^N$ (since $z$ is in the closure of the complement of $L(A,\delta/3^N)$) as well as $|B(z)|\geq \delta/3^K$ (by~\eqref{eq:separation formula}).  Similar considerations apply to $E_B$.	
\end{Remark}

An immediate consequence of Lemma~\ref{le:separation} is the equivalence of the quantities $\delta(A,B)$ and $\widetilde{\delta}(A,B)$.

\begin{Corollary}\label{co:equivalence of deltas}
	If $\widetilde{\delta}(A,B)$ is defined by~\eqref{eq:tilde delta},	
then 
\begin{equation}\label{eq:equiv of deltas}
  \frac{1}{3^{\max(N,K)}}  \delta(A, B)  \le
	\widetilde\delta(A,B)\le \delta(A, B).
\end{equation}
\end{Corollary}

\begin{proof}
	The inequality on the right hand side of \eqref{eq:equiv of deltas} is immediate.	
	For the left hand side, note that~\eqref{eq:separation formula} is equivalent to
	\[
	\Big\{z \in \C: |A(z)|\ge \frac{\delta}{3^N}\Big\}\bigcup
	\Big\{z \in \C: |B(z)|\ge \frac{\delta}{3^K}\Big\}=\C.
	\]
	This implies that any $z\in\C$ must satisfy either $|A(z)|\ge \delta/3^N$ or $|B(z)|\ge \delta/3^K$. Therefore,
	\[
	\widetilde\delta(A,B)\ge \frac{\delta(A,B)}{3^{\max(N,K)}},
	\]
	which is precisely the left hand side of~\eqref{eq:equiv of deltas}.
\end{proof}

We have noted in Section 2 that $\widetilde{\delta}$ is bounded on $\mathfrak{B}$ defined by~\eqref{eq:product ball}. From Corollary 3.11 it follows that the same is true for $\delta$, and so there exists a $T>0$, depending only on $N$ and $K$, such that
\begin{equation}\label{eq:boundedness of delta}
	\delta(A,B) \le T
\end{equation}
for all $A\in\C_N[z]$, $B\in\C_K[z]$ with $\| A \|\le1$, $\| B \|\le 1$.

Lemma~\ref{le:separation} will be used in the sequel to construct the systems of contours that appear in a Cauchy integral method to compute the coefficients of the polynomials  $R$ and $S$.

\section{The Cauchy integral method}

The following lemma is implicit  in~\cite[Section 1.3]{Yger}, but we  include a direct proof in order to make our paper self-contained.

\begin{Lemma}\label{le:general formula}
	Suppose $\Omega\subset\C$ is an open set, $F,G$ are two analytic functions defined on $\Omega$, $\alpha\in \Omega$  is a zero of $G$, and $\Gamma$ is a positively oriented system of closed rectifiable contours in $\Omega\setminus\{\alpha\}$ whose index with respect to $\alpha$ is 1, and such that $G(\zeta)\not=0$ for any $\zeta\in\Gamma$.  Define 
	\[
	g(\zeta, z)=\frac{G(\zeta)-G(z)}{\zeta-z} \quad \mbox{and} \quad 
	\phi(z)=\frac{1}{2\pi i} \int_\Gamma \frac{g(\zeta, z)F(\zeta)}{G(\zeta)} \, d\zeta.
	\]
	Then:
	\begin{enumerate}
		\item $\phi$ is analytic on $\Omega\setminus\Gamma$ and $\phi(\alpha)=F(\alpha)$.
		
		\item If $G \in \C_{N}[z]$, then $\phi \in \C_{N-1}[z]$.
		
	\end{enumerate}
	
\end{Lemma}

\begin{proof}
	(i): For $\zeta\in \Gamma$, the function $z \mapsto g(\zeta, z)$ is analytic on $\Omega\setminus\Gamma$. This implies that $\phi$ is also analytic on $\Omega\setminus\Gamma$.
	 Since $G(\alpha)=0$, we see from Cauchy's formula that 
	\[
	\phi(\alpha)=\frac{1}{2\pi i} \int_\Gamma \frac{G(\zeta)F(\zeta)}{G(\zeta)(\zeta-\alpha)} \, d\zeta=\frac{1}{2\pi i} \int_\Gamma \frac{F(\zeta)}{\zeta-\alpha} \, d\zeta
	=F(\alpha).
	\] 
	
	(ii): 
	If $G \in \C_{N}[z]$, then for every fixed $\zeta\in\Gamma$, $g(\zeta, \cdot) \in \C_{N-1}[z]$, which implies that $\phi \in \C_{N-1}[z]$.	
\end{proof}

\begin{Remark}\label{re:development}
So that we can use them later, we now compute the coefficients of $\phi \in \C_{N - 1}[z]$ when $G(z)=g_0+g_1 z+\dots+g_Nz^N$.
  Observe that
\begin{align*}
	G(\zeta)-G(z)&=\sum_{k=1}^Ng_k(\zeta^k-z^k)\\
	&=(\zeta-z)\sum_{k=1}^Ng_k\frac{\zeta^k-z^k}{\zeta-z}\\
	&=(\zeta-z)\sum_{k=1}^Ng_k\sum_{j=0}^{k-1}z^j\zeta^{k-1-j}\\
	&=(\zeta-z)\sum_{j=0}^{N-1}z^j\sum_{k=j+1}^N g_k\zeta^{k-(j+1)},
\end{align*}
and hence
\[
g(\zeta,z)=\sum_{j=0}^{N-1}z^j\sum_{k=j+1}^N g_k\zeta^{k-(j+1)}.
\]
Thus,
\begin{equation}\label{Form of phi}
	\phi(z)=\sum_{j=0}^{N-1}z^j\sum_{k=j+1}^N g_k\frac{1}{2\pi i}\int_{\Gamma}\frac{F(\zeta)\zeta^{k-(j+1)}}
	{G(\zeta)}d\zeta,
\end{equation}
where $\Gamma$ is a system of contours as in Lemma \ref{le:general formula}.
\end{Remark}


Next we obtain the solution to a B\'ezout-type equation when the zeros of $A$ and $B$ are simple. For $\zeta, z \in \C$, define 
\[
a(\zeta, z)=\frac{A(\zeta)-A(z)}{\zeta-z} \; \;  \mbox{and} \; \;  
b(\zeta, z)=\frac{B(\zeta)-B(z)}{\zeta-z}, \; \; \zeta \not = z,
\]
and of course, $a(\zeta, \zeta) = A'(\zeta)$, $b(\zeta, \zeta) = B'(\zeta)$.

\begin{Corollary}\label{co:Bezout extended}
	Let  $A, B, P$ be polynomials with $\deg A=N$, $\deg B=K$, and $\deg P\le N+K-1$. Suppose the roots $\{\alpha_{i}\}_{i = 1}^{N}$ of $A$ and the roots $\{\beta_{j}\}_{j = 1 }^{K}$ of $B$ are all simple and $\{\alpha_i\}_{i = 1}^{N} \cap \{\beta_j\}_{j = 1}^{K} = \varnothing$.  Consider the equation
	\begin{equation}\label{eq:bezout extended}
		A(z)\widehat R(z)+B(z)\widehat S(z)=P(z).
	\end{equation}
Suppose that $\Gamma_1, \Gamma_2$ are positively oriented systems of closed rectifiable contours in $\C$ such that $\Gamma_1$ surrounds the roots of $A$ (with index $1$) and not those of $B$, while the opposite is true for $\Gamma_2$.
Then the  unique solutions $\widehat R, \widehat S$ of~\eqref{eq:bezout extended} with the property that $\deg \widehat R\le K-1$ and $\deg \widehat S\le N-1$ are given by the formulas
\begin{equation}\label{eq:def of S_p, R_p}
		\widehat S(z):=\frac{1}{2\pi i} \int_{\Gamma_1} \frac{a(\zeta, z)P(\zeta)}{A(\zeta)B(\zeta)} \, d\zeta, \quad
			\widehat R(z):=\frac{1}{2\pi i} \int_{\Gamma_2} \frac{b(\zeta, z)P(\zeta)}{A(\zeta)B(\zeta)} \, d\zeta.
\end{equation}

\end{Corollary}

\begin{proof}
	Apply
	 Lemma~\ref{le:general formula}  to the functions $G(z)=A(z)$ and $F(z)= P(z)/B(z)$, the region $\Omega=\C\setminus\{\beta_1, \dots, \beta_K\}$, and  the contour $\Gamma_1$. The first formula in~\eqref{eq:def of S_p, R_p} yields 
	 a polynomial $\widehat S(z) \in \C_{N - 1}[z]$ such that 
	 \begin{equation}\label{eq:S}
	 	\widehat S(\alpha_i)=\frac{P(\alpha_i)}{B(\alpha_i)} \; \mbox{ for all $1 \leq i \leq N$.}
	 \end{equation}

 Similarly, by the second formula in~\eqref{eq:def of S_p, R_p}, one obtains $\widehat R(z) \in \C_{K - 1}[z]$ such that
	  \begin{equation}\label{eq:R}
	 	\widehat R(\beta_j)=\frac{P(\beta_j)}{A(\beta_j)} \; \mbox{ for all $1 \leq j \leq K$.}
	 \end{equation}

	One can check that $\widehat RA+\widehat SB \in \C_{N+K-1}[z]$ which takes the values $P(\alpha_i)$ at the roots of $A$  and $P(\beta_j)$ at the roots of $B$. The assumption of simplicity of the roots of $A$ and $B$ implies that the total number of roots of $A$ and $B$ is precisely $N+K$. Therefore $\widehat RA+\widehat SB$ takes the same values as $P$ at $N+K$ distinct points. Since $\widehat RA+\widehat SB$ and $P$ are both polynomials of degree at most $N+K-1$, they must coincide everywhere. This proves the existence of $\widehat{R}$ and $\widehat{S}$. 
	 
	 To prove uniqueness, suppose that $\widehat R_1, \widehat S_1$ were another pair of polynomials satisfying the same conditions. Then
	 \[
	 A(z)(\widehat R(z)-\widehat R_1(z))+B(z)(\widehat S(z)-\widehat S_1(z))=0 \; \mbox{for all $z\in\C$.}
	 \]
	 For any $1 \leq i \leq N$, observe that $A(\alpha_i)=0$ while $B(\alpha_i)\not=0$. Therefore,  $\widehat S(\alpha_i)=\widehat S_1(\alpha_i)$. Since $\widehat S, \widehat S_1 \in \C_{N - 1}[z]$, they must coincide everywhere. A similar argument shows that  $\widehat R=\widehat R_1$, which completes the proof. 
	 	\end{proof}

In particular, when $P(z)=z^t$ for some integer $ 0 \leq t\le N+K-1 $, we see, using~\eqref{eq:def of S_p, R_p} as well as formula~\eqref{Form of phi}, that the (unique) solutions $\widehat S$ and $\widehat R$ (of appropriate degree)  of the B\'{e}zout type identity
\begin{equation}\label{eq:bezout with arbitrary monomial}
A(z)\widehat R(z)+B(z)\widehat S(z)=z^t
\end{equation}
are given by 
\begin{equation}\label{eq:solutions for arbitrary monomial}
	\begin{split}
\widehat  S(z)&=\sum_{j=0}^{N-1}z^j\sum_{k=j+1}^N a_k\frac{1}{2\pi i}\int_{ {\Gamma}_1}\frac{\zeta^{t+k-j-1}}
{ {A}(\zeta) {B}(\zeta)}\,d\zeta,\\
 \widehat R(z)&=\sum_{j=0}^{K-1}z^j\sum_{k=j+1}^K  {b}_k\frac{1}{2\pi i}\int_{ {\Gamma}_2}\frac{\zeta^{t+k-j-1}}
{ {A}(\zeta) {B}(\zeta)}\,d\zeta.
\end{split}
\end{equation}

\section{A weak form of the main result}\label{se:weak form}

In this section we show how
  the description of the solutions of~\eqref{eq:bezout with arbitrary monomial}, given by \eqref{eq:solutions for arbitrary monomial}, may be used to obtain estimates for the coefficients of the polynomials $R$ and $S$ in~\eqref{Bezout_C}. Although these estimates are not sufficient to yield Theorem~\ref{th:main}, since they depend on the size of the roots of $ A $ and~$ B $, the proof is much simpler. This weaker version gives a feeling of the general case and is enough to give the bounds on the norm of the inverse of the Sylvester matrix (see Theorem \ref{th:estimates on Sylvester} below).

Recall that we consider
two polynomials
$A$ and $B$ of respective degrees $N$ and $K$, that is, $A(z)=\sum_{k=0}^{N}a_kz^k$, $B(z)=\sum_{k=0}^{K}b_kz^k$. We also assume $\|A\|, \| B \|\le 1$. The roots of $A$ are denoted by $\{\alpha_i\}_{i = 1}^{N}$ and the roots of $B$ by $\{\beta_j\}_{j = 1}^{K}$.  We start with a lemma which will allow us to reduce to the case when  both $A$ and $B$ have simple roots.

\begin{Lemma}\label{le:reduction}
Let $A\in \C_N[z]$, $B\in\mathbb C_K[z]$   satisfy $\|A\|,\|B\|\leq 1$. Then there exist  $A_n\in\C_N[z], B_n\in\C_K[z]$, all with distinct roots, such that
\begin{enumerate}
\item  $\|A_n\|\le 1$ and $\|B_n\|\le 1$ for all $n$; 
\item $A_n\to A$, $B_n\to B$, and 
\item  $\delta(A_n, B_n)\to \delta(A,B)$.
\end{enumerate}
\end{Lemma}

\begin{proof}
Write 
\[
A(z)=a_N(z-\alpha_1)\dots (z-\alpha_N)\; \; \mbox{and}\; \; B(z)=b_K(z-\beta_1)\dots (z-\beta_K),
\]	
where we do not assume that the roots $\alpha_1,\alpha_2,\dots,\alpha_N$ are different and similarly for $\beta_1,\beta_2,\dots\beta_K$.
Take sequences $\alpha_1^{(n)},\alpha_2^{(n)},\dots,\alpha_N^{(n)}$ such that for all $1\leq i\leq N$, $\alpha_i^{(n)}\to \alpha_i$ when $n\to\infty$ and for all $1\leq i,s\leq N$, $i\neq s$, and $n\geq 1$, $\alpha_i^{(n)}\neq \alpha_s^{(n)}$. Similarly, take sequences $\beta_1^{(n)},\beta_2^{(n)},\dots,\beta_N^{(n)}$ such that for all $1\leq j\leq K$, $\beta_j^{(n)}\to \beta_j$ when $n\to\infty$ and for all $1\leq j,\ell\leq K$, $j\neq \ell$, and $n\geq 1$, $\beta_j^{(n)}\neq \beta_\ell^{(n)}$. If 
\[
A_n^\sharp(z):=a_N(z-\alpha_1^{(n)})\dots (z-\alpha_N^{(n)}) \; \; \mbox{and} \; \; B_n^\sharp(z):=b_K(z-\beta^{(n)}_1)\dots (z-\beta^{(n)}_K),
\]
define 
$$A_n=\frac{\|A\|}{\|A_n^\sharp \|}A_n^\sharp \quad \mbox{and} \quad B_n=\frac{\|B\|}{\|B_n^\sharp\|}|B_n^\sharp.$$

By construction, $A_n$ and $B_n$ have both simple roots and satisfy $\|A_n\|\le 1$ and $\|B_n\|\le 1$. Moreover, by Viete's formula,  $\|A_n-A\|\to 0$ and $\|B_n-B\|\to 0$ as $n\to\infty$. Finally, from 
$\alpha_i^{(n)}\to \alpha_i$ and $\beta_j^{(n)}\to \beta_j$ it follows that $\delta(A_n,B_n)\to \delta(A,B)$.
\end{proof}

We now need to find appropriate contours $\Gamma_1, \Gamma_2$ to apply to Corollary \ref{co:Bezout extended}.
Returning   to the notation in   Lemma~\ref{le:separation}, we have
\begin{equation}\label{eq:L and E}
L(A,\frac{\delta}{3^N})\subset E_A \; \; \mbox{and} \; \;  L(B,\frac{\delta}{3^K})\cap E_A=\varnothing. 
\end{equation}
In particular, $\alpha_i\in E_A$ for all $1 \leq i \leq N$.

In the definition of $ E_A $, we may apply the distributivity of the (outer) intersection with respect to the (inner) union to obtain
\begin{equation}\label{eq:distributivity}
	E_A=\bigcup_{1\le i_1,\dots, i_K\le N } 	\Omega_{i_1i_2\dots i_K}, 
\end{equation}
where
\begin{equation}\label{eq:the intersections}
	\Omega_{i_1i_2\dots i_K} := \bigcap_{j=1}^K D(\alpha_{i_j}, \tfrac{1}{3}|\alpha_{i_j}-\beta_j|).
\end{equation}
As the $N^K$ sets $	\Omega_{i_1i_2\dots i_K}$ are intersections of disks, the boundary of $ E_A$ is formed by a finite number of circular arcs and is therefore a rectifiable system of closed contours. Let $ \operatorname{len} $ denote the length of such a curve. Since 
\[
\partial E_A\subset \bigcup_{i_1, \dots, i_K} \partial \Omega_{i_1i_2\dots i_K},
\]
we have
\begin{equation}\label{eq:length partial E_A}
	\operatorname{len} (\partial E_A)\le \sum_{i_1, \dots, i_K} \operatorname{len}(\partial\Omega_{i_1i_2\dots i_K} ).
\end{equation}
Define $\Gamma_1$ to be $\partial E_A$ oriented such that its index with respect to points in $E_A$ is 1, while it is 0 with respect to points outside~$\overline{E_A}$ (see Figure \ref{Fig_ALL2}). In particular, this applies to the points $\alpha_i\in E_A$ and $\beta_j\notin\overline{E_A}$ (recall Remark~\ref{re:form of E_A}).

In a similar way, $\Gamma_2$ will be the boundary of $E_B$, appropriately oriented (see Figure \ref{Fig_ALL2}).
Using Remark \ref{re:form of E_A} again, it follows that 
if $\zeta\in\Gamma_1\cup\Gamma_2$, then $|A(\zeta)|\ge \delta/3^N$ and $|B(\zeta)|\ge \delta/3^K$. 

In order to estimate the length of the contours discussed above, we will need the following folklore result (see \cite[Ch. 1]{MR0123962} for a precise reference).

\begin{Lemma}\label{le:convex}
	If $G_1$ and $G_2$ are open convex sets with $G_1\subset G_2$, then $\operatorname{len}(\partial G_1) \leq \operatorname{len}(\partial G_2)$.
\end{Lemma}

\begin{figure}
\begin{center}
 \includegraphics[width=.5\textwidth]{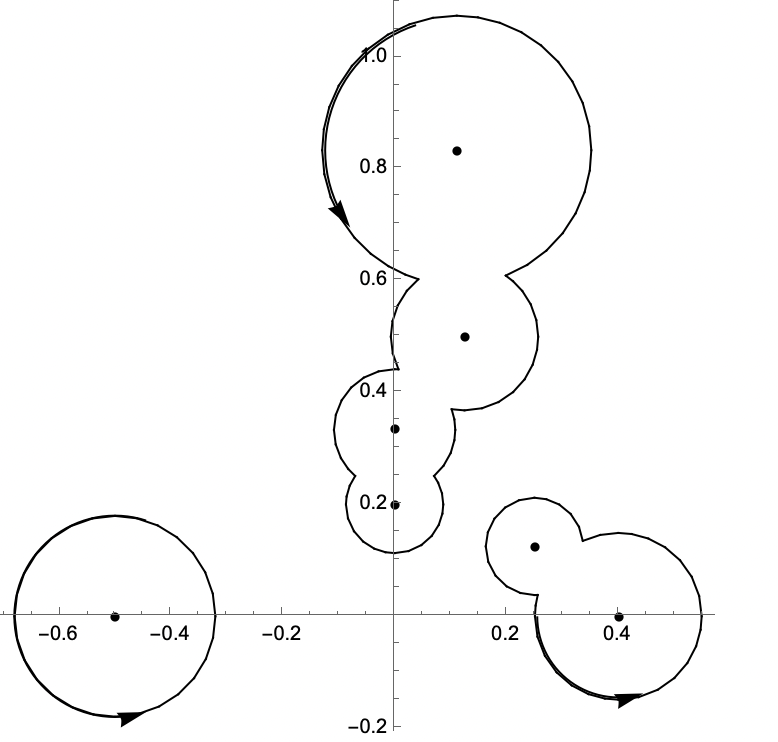}
 \caption{The sets $E_A$ and $E_B$ from Figure~\ref{figure 1} with boundaries oriented.  The contour $\Gamma_1 = \partial E_{A}$ (bottom --  in two pieces) surrounds the zeros of $A$, while $\Gamma_2 = \partial E_{B}$ (top - in one piece) surrounds the zeros of $B$. Recall that $A$ is the monic polynomial whose roots are $(\alpha_1, \alpha_2, \alpha_3) = (\tfrac{1}{4} + \tfrac{i}{8}, -\tfrac{1}{2}, \tfrac{2}{5})$ and $B$ is the monic polynomial whose zeros are $(\beta_1, \beta_2, \beta_3, \beta_4) = (\tfrac{1}{9} + \tfrac{5}{6}i, \tfrac{1}{8} + \tfrac{i}{2}, \tfrac{i}{3}, \tfrac{i}{5})$.}
 \label{Fig_ALL2}
 \end{center}
\end{figure}

\begin{Lemma}\label{co:basic estimate} For all $1 \leq i \leq N$ and $1 \leq j \leq K$, suppose that $|\alpha_i|, |\beta_j|\le M$ for some $M>0$. Fix $\ell\in \N$.
	With the notation above, there is a constant $C_1 > 0$, depending only on $N$ and $K$, such that
	\begin{equation}\label{eq:basic estimate}
	\left|\int_{\Gamma_s}\frac{\zeta^{\ell}}
	{A(\zeta)B(\zeta)}d\zeta\right|\le \frac{C_1 M^{\ell + 1}}{\delta^2}, \quad s = 1, 2.
	\end{equation}
	
\end{Lemma}

\begin{proof} 
	
	The hypothesis says that $|\alpha_i-\beta_j|\le 2M$ for all $i$ and $j$, and so $E_A \subset D(0, 5M/3)$. By Lemma~\ref{le:convex}, we have
	$$
	\operatorname{len}(\partial\Omega_{i_1i_2\dots i_K})\le 2\pi \frac{5M}{3} \; \; \mbox{for every $i_1, \dots, i_K$.}
	$$
	Then~\eqref{eq:length partial E_A} yields 
	\[
	\operatorname{len}(\Gamma_1)\le \frac{10\pi MN^K}{3},
	\]
and therefore, 
	\[
	\left|\int_{\Gamma_1}\frac{\zeta^{\ell}}
	{A(\zeta)B(\zeta)}d\zeta\right|\le \frac{10\pi MN^K }{3}
	\left(\frac{5M}{3}\right)^{\ell}\frac{3^{N + K}}{\delta^2}.
	\]
	This proves \eqref{eq:basic estimate} for $\Gamma_1$. A similar argument is used for $\Gamma_2$.
\end{proof}

As a consequence, one obtains a weaker form of Theorem~\ref{th:main}, where the bounds on the coefficients of $R$ and $S$  depend on the size of the zeros of $ A,B $.

\begin{Theorem}\label{co:main theorem weak}
	Let $A$ and $B$ be two polynomials of respective degrees $N$ and $K$. Suppose that $ \| A \|, \| B \|\le 1 $ and their roots are bounded in modulus by $ M>0 $.  
	If $ \delta(A,B)>0$, then there is $ C(N,M,K)>0 $ such that the unique minimal 
	polynomials 
	$R \in \C_{K - 1}[z]$ and $S \in \C_{N - 1}[z]$ from \eqref{Bezout_C} satisfy
	\begin{equation*} 
	\|R\|,\|S\|\le \frac{C(N,K,M)}{\delta^2}.
	\end{equation*}
\end{Theorem}

\begin{proof}
By Lemma \ref{le:sylvester} and Lemma \ref{le:reduction} we can assume that the roots of $A$ and $B$ are simple. 
	From~\eqref{eq:solutions for arbitrary monomial} (when $ t=0 $), it follows that 
	$
	S(z)=\sum_{j=0}^{N-1} s_jz^j,
	$
	with
	\[
	s_j=\sum_{k=j+1}^Na_k \frac{1}{2\pi i} \int_{\Gamma_1} \frac{\zeta^{k-j-1}}{A(\zeta)B(\zeta)} d\zeta.
	\]
Since $ |a_k|\le 1$, Lemma~\ref{co:basic estimate} yields the desired estimate for $ \| S \| $. A similar argument estimates $ \| R \| $.	
\end{proof}

It is worth pointing out that simple examples show that even though $\|A\|$ and $\|B\|$ are bounded by $1$, their roots can be arbitrarily large.

Before proceeding with the next two sections, it is worth reminding the reader of our standing assumption that the constants $C_1, C_2, \dots$ that will appear in the estimates below always depend {\em only on $N$ and $K$} and {\em not on the coefficients of} $A \in \C_{N}[z]$ and $B \in \C_{K}[z]$.

\section{Intermezzo---the Sylvester matrix}\label{se:Sylvester}

In this section, we show how Lemma~\ref{co:basic estimate} can be used  to obtain an estimate of the norm of the inverse of the Sylvester matrix. 
To see this,  consider a general polynomial 
$$P(z)=p_0+p_1z+\dots+p_{N+K-1}z^{N+K-1}$$ and the B\'{e}zout type equation 
\begin{equation}\label{eq:AR+BS=P}
A(z)R(z)+B(z)S(z)=P(z).
\end{equation}
Suppose that, as above, $A$ and $B$ have no common roots. Using the notation in Section~\ref{se:prelim}, the identity in \eqref{eq:AR+BS=P} translates into the system of $N+K$ equations
\begin{equation}\label{eq:system for coefficients}
\mathfrak{S}(A,B) \bf x=\bf p,
\end{equation}
where $\mathfrak{S}(A,B)$ is the $(N + K) \times (N + K)$ {\em Sylvester matrix}
\begin{equation}\label{Syldslfgdf}
\mathfrak{S}(A,B)=
\left[
\begin{array}{rrrrrrrr}
a_0&&&&b_0&&&\\
a_1&a_0&&&b_1&b_0&&\\
a_2&a_1&\ddots&&b_2&b_1&\ddots&\\
\vdots&&\ddots&a_0&\vdots& &\ddots&b_0\\
&\vdots& &a_1& &\vdots& &b_1\\
a_N&&&&b_K&&&\\
&a_N&&\vdots&&b_K&&\vdots\\
&&\ddots&&&&\ddots&\\
\undermat{K \text{ columns}}{&&&a_N}&\undermat{N\text{ columns}}{&&&b_K}
\end{array}\right],
\end{equation}

\bigskip\bigskip
\noindent 
$$\mathbf{x} = [r_0, r_1, \cdots, r_{K - 1}, s_{0}, s_1, \cdots, s_{N - 1} ]^{T},$$
and 
$$\mathbf{p} = [p_0, p_1,  \ldots, p_{N+K-1}]^{T}.$$

\begin{Theorem}\label{th:estimates on Sylvester} 
Let $A$ and $B$ are two polynomials of respective degree $N$ and $K$, with $ \| A \|, \| B \|\le 1 $, $ \delta=\delta(A,B) $ is given by~\eqref{eq:min values A B}, and let 
$$M=\max(\|A\|/|a_N|,\|B\|/|b_K|).$$ There is a constant $ C_3>0$ (depending only on $N$ and $K$) such that if $\mathfrak{S}= \mathfrak{S}(A,B) $ denotes the Sylvester matrix of $ A $ and $ B $, then
	\begin{equation}\label{eq:estimates Sylvester }
	\| \mathfrak{S}^{-1} \|\le \frac{C_3 M^{N+K+\max(N,K)-1}}{\delta^2}. 	
	\end{equation}
\end{Theorem}

\begin{proof}
By Lemma \ref{le:sylvester} and Lemma \ref{le:reduction} we can assume that the roots of $A$ and $B$ are simple. 	
	
	From a classical result of Cauchy (for instance, see \cite[p.~167]{MR1894714}) it follows that the roots of $ A $ (denoted by $\{\alpha_i\}_{i = 1}^{N}$) and $ B $ (denoted by $\{\beta_j\}_{j = 1}^{K}$) satisfy the inequality
	\begin{equation}\label{eq:cauchy roots}
	|\alpha_i|, |\beta_j|\le 1+M\leq 2M.
	\end{equation}	
	
	For each $0 \leq \ell \leq N+K-1$, let $R^{(\ell)}, S^{(\ell)}$ denote the solution of~\eqref{eq:system for coefficients} corresponding to $P(z)=z^\ell$. If $\mathfrak{R}$ denotes the $(N+K)\times (N+K)$ matrix with columns
	\[
	[r^{(\ell)}_0, \dots, r^{(\ell)}_{K-1}, s^{(\ell)}_0, \dots, s^{(\ell)}_{N-1} ]^T,\quad 0 \leq \ell \leq  N+K-1,
	\]
	it follows from~\eqref{eq:system for coefficients} that $ \mathfrak{S}\mathfrak{R}=I_{N+K} $, and therefore $ \mathfrak{R}= \mathfrak{S}^{-1} $.
	
	To prove the theorem, we need to estimate the coefficients $ r^{(\ell)}_i $ and $ s^{(\ell)}_j$, for $0 \leq \ell \leq N+K-1$, $0\leq i\leq N-1$ and $0\leq j\leq K-1$. If we fix $0\leq \ell\leq N+K-1 $, applying the formulas from \eqref{eq:solutions for arbitrary monomial}, we have 
\[
 s^{(\ell)}_j=\sum_{k=j+1}^N a_k \frac{1}{2i\pi}\int_{\Gamma_1}\frac{\zeta^{\ell+k-j-1}}{A(\zeta)B(\zeta)}\,d\zeta.
\]	
It follows from  Lemma~\ref{co:basic estimate} that
\[
| s^{(\ell)}_j|\leq \sum_{k=j+1}^N |a_k|\frac{C_1}{\delta^2}(2M)^{\ell+k-j}\leq  \frac{C_1}{\delta^2} \sum_{k=j+1}^N (2M)^{\ell+k-j}\leq \frac{NC_1}{\delta^2} (2M)^{\ell+N-j}. 
\]
Now use the fact that $\ell+N-j\leq N+K+\max(N,K)-1$ and $M\geq 1$, which gives that 
\[
| s^{(\ell)}_j|\leq \frac{C_3 M^{N+K+\max(N+K)-1}}{\delta^2}.
\]

Similar estimates hold for $| r^{(\ell)}_i|$, and we thus get \eqref{eq:estimates Sylvester }. 	 
	 Notice the use of the fact that $\|\mathfrak{S}^{-1}\|$ is equivalent to the maximum of its entries $|r_{i}^{(\ell)}|$ and $|s_{j}^{(\ell)}|$ since all norms on a finite dimensional Banach space are equivalent.
\end{proof}

\begin{Remark}
If one removes the hypothesis of $\|A\|, \|B\| \leq 1$, one can make small adjustments to the proof of Theorem \ref{th:estimates on Sylvester} to obtain the estimate
$$\| \mathfrak{S}^{-1} \|\le \frac{C_3 M^{N+K+\max(N,K)-1} \max(\|A\|, \|B\|)}{\delta^2}. $$	
\end{Remark}

\begin{Remark}
It is natural for the estimate of $\|\mathfrak{S}^{-1}\|$ to depend on the leading coefficients $|a_N|$ and $|b_K|$. To see this, take $A(z)=az, B(z)=az+1$. Then $\delta = 1$, and 
\[\mathfrak{S} = \begin{bmatrix} 0 & 1\\ a & a \end{bmatrix}, \quad
\mathfrak{S}^{-1}=\begin{bmatrix} -1 & a^{-1}\\ \phantom{-}1 & 0\end{bmatrix}.
\]
Thus, $\big\|\mathfrak{S}^{-1} \big\| $
  becomes unbounded as $a \to 0$. Note that this example also shows that for $t>0$ one cannot expect bounds for $\widehat R$ and $\widehat S$ from \eqref{eq:bezout with arbitrary monomial} in terms only of $\delta$.
 
\end{Remark}

\section{Finding convenient contours}\label{se:contours}

We return now to our main purpose, that of 
proving Theorem~\ref{th:main}. To obtain better estimates than those given in Section~\ref{se:weak form},  we need to use Corollary~\ref{co:Bezout extended} more delicately.
  The most important step is an appropriate choice of the contours $\Gamma_1$ and $\Gamma_2$. This is the goal of this section. We will only discuss the construction for $\Gamma_1$ since the construction of $\Gamma_2$ is analogous.

	For sufficiently small 
$\eps$, depending on $N$ and $K$, there are $N+K+1$ disjoint closed disks of radius $2\eps$ inside $\D$. It follows that at least one of these disks does not contain any of the roots of $A$ or $B$. Indeed, suppose it is 
$\overline{D(z_0, 2\eps)}$ for a suitable $z_0\in \D$.

Now replace $A(z)$ and $B(z)$ by $A(z+z_0)$ and $B(z+z_0)$ and notice that 
$$\delta(A, B) = \delta(A(\cdot + z_0), B(\cdot + z_0)),$$ while
\begin{equation}\label{eq:after translation}
\| A(\cdot+z_0) \|, \| B(\cdot +z_0) \|\le C_4,
\end{equation}
where $C_4$ depends only on $N,K$. From now on, we will assume that $\overline{D(0, 2\eps)}$ does not contain any of the zeros of $A$ or $B$.

Also observe that the equivalence of norms on finite dimensional Banach spaces (with constants depending only
on the dimension), applied to $\C_{N}[z]$, yields
\[
 \|A\|_{\eps}:= \sup_{|z|\le \eps}|A(z)|\ge C_5 \|A\|
\]
for some $C_5>0$. By the maximum principle, one can choose a $z_1\in  \partial D(0,\eps)$ such that 
$$|A(z_1)|= \sup_{|z|\le \epsilon}|A(z)|.$$

If $z \in \C$ satisfies  $|z-\alpha_i|\ge\tfrac{1}{2} |\alpha_i-z_1|$ for every $1 \leq i \leq N$, then the previous inequality shows that 
\begin{equation} \label{estimA1}
	|A(z)| =\Big|a_N\prod_{i=1}^N(z-\alpha_i)\Big|
	\ge \Big|a_N\prod_{i=1}^N \tfrac{1}{2}(z_1-\alpha_i)\Big|
	 =|A(z_1)| 2^{-N}
	 \ge C_5\|A\| 2^{-N}.
\end{equation}
Since $|z_1|=  \eps \le \tfrac{1}{2}|\alpha_i|$, we have $\tfrac{1}{2}|\alpha_i-z_1|\le \tfrac{1}{2}(|\alpha_i|+\eps)\le \frac{3}{4}|\alpha_i|$. 

Let 
$$D_i :=D\big(\alpha_i,\tfrac{3}{4}|\alpha_i|\big) \; \;  \mbox{and} \; \; 
 D_A :=\bigcup_{i = 1}^{N} D_i.$$  Each $z\not\in D_A$ satisfies 
 $$|z-\alpha_i|\ge  \tfrac{1}{2}|\alpha_i-z_1| \; \mbox{for every $1 \leq i \leq N$}.$$ From \eqref{estimA1} it follows that $L(A,C_5\|A\|2^{-N})\subset D_A$ (Recall the definition of the sub-level set $L(A, \cdot)$ from \eqref{LLLLLL}).

Since $|\alpha_i|>2\eps$ for all $1 \leq i \leq N$, the triangle inequality says that if $ z\in D_i $, then $ |z|\ge \tfrac{1}{4}|\alpha_i| \geq\tfrac{1}{2}\eps $. Thus
\begin{equation}\label{eq:z>eps/2}
	|z|\ge \frac{\eps}{2} \text{ for all } z\in \overline{D_{A}} .
\end{equation}

Now define the sets 
\[
F_A^{(i)}:= E_A\cap D_i \quad \mbox{and} \quad 
 F_A:=E_A\cap D_A= \bigcup_{i=1}^N F_A^{(i)}.
\]
In light of \eqref{eq:the intersections}, $ F_A^{(i)} $ is a union of at most $ N^{K} $ intersections of disks. Each of these intersections is contained in $ D_i $ and thus Lemma~\ref{le:convex} says that the length of its boundary is bounded above  by $ \operatorname{len}(\partial D_i)=2\pi \frac{3}{4}|\alpha_i| $. Therefore,
\begin{equation}\label{eq:length of }
\operatorname{len}(\partial F_A^{(i)})\le \frac{3\pi N^K}{2}|\alpha_{i}| \; \; \mbox{for all $1 \leq i \leq N$}.
\end{equation}

We define the new system of contours $\Gamma_1$ by
\[
\Gamma_{1} :=\partial  (D_A\cap E_A).
\]
Since $ \alpha_i\in D_i $, and we already know  that $ \alpha_i\in E_A $ (see Remark~\ref{re:form of E_A}), it follows that $ \alpha_i\in D_A\cap E_A $ for all $ i $. On the other hand, $ \beta_j\notin \overline{E_A}$ (also by Remark~\ref{re:form of E_A}) and so $ \beta_j\notin\overline{D_A\cap E_A} $ for all $ j $. Therefore, we may orient $ \Gamma_1 $ such that its index is $ 1 $ for any $ \alpha_i $ and $ 0$ for any $ \beta_j $
(An example appears in Figures \ref{Fig_EADA} and   \ref{Fig_Gamma1}).
\begin{figure}
\begin{center}
 \includegraphics[width=.7\textwidth]{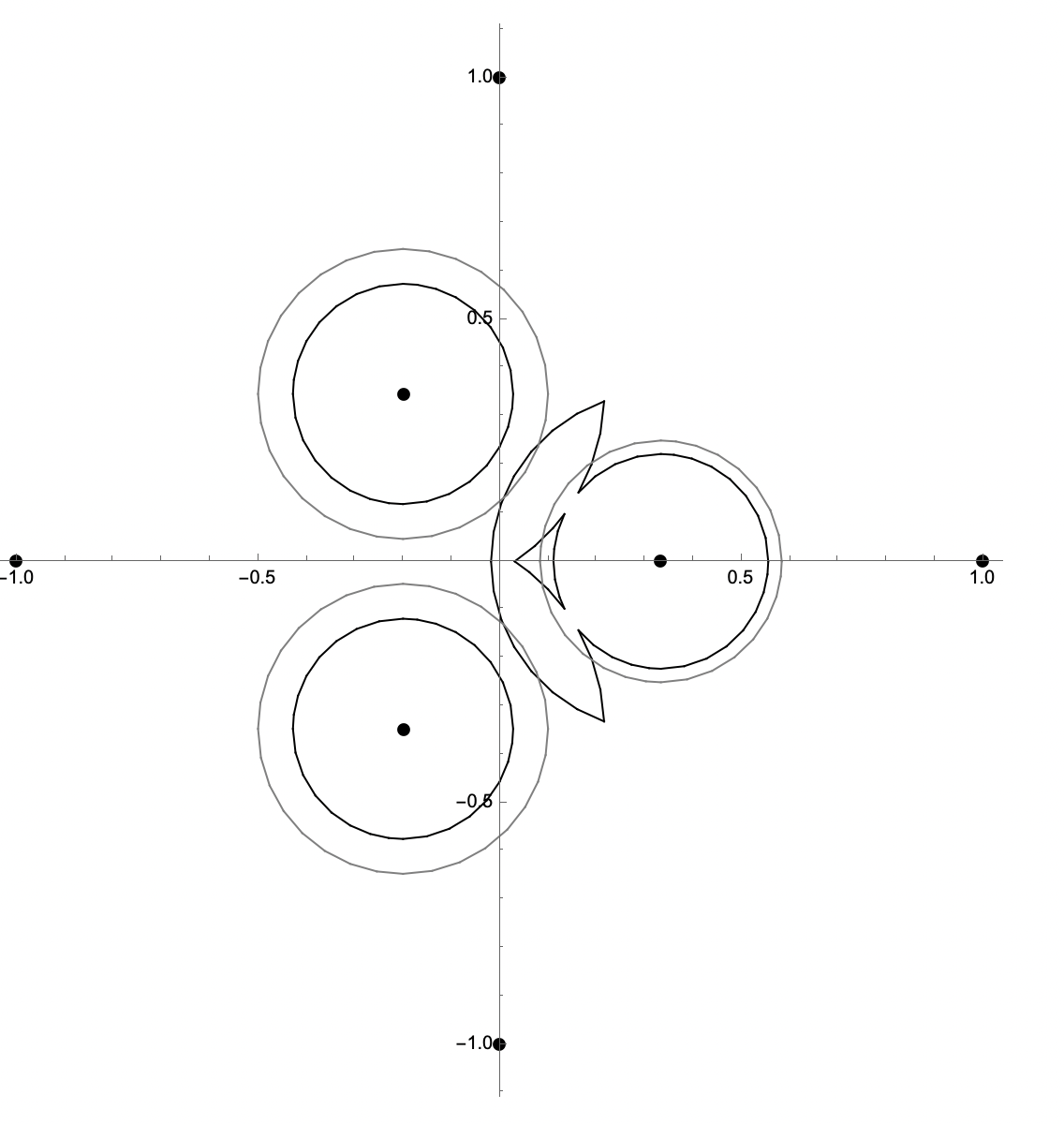}
 \caption{The regions $E_A$ (black) and $D_A$ (gray) with the zeros  $\{\alpha_i\}_{i = 1}^{3}$ of $A$ and the zeros $\{\beta_j\}_{j = 1}^{4}$ of $B$. Here $(\alpha_1, \alpha_2, \alpha_3) = (\tfrac{1}{3}, -0.2+0.34641 i, -0.2-0.34641 i)$ and $(\beta_1, \beta_2, \beta_3, \beta_4) = (1, i, -1, -i)$.}
 \label{Fig_EADA}
 \end{center}
\end{figure}
\begin{figure}
\begin{center}
 \includegraphics[width=.7\textwidth]{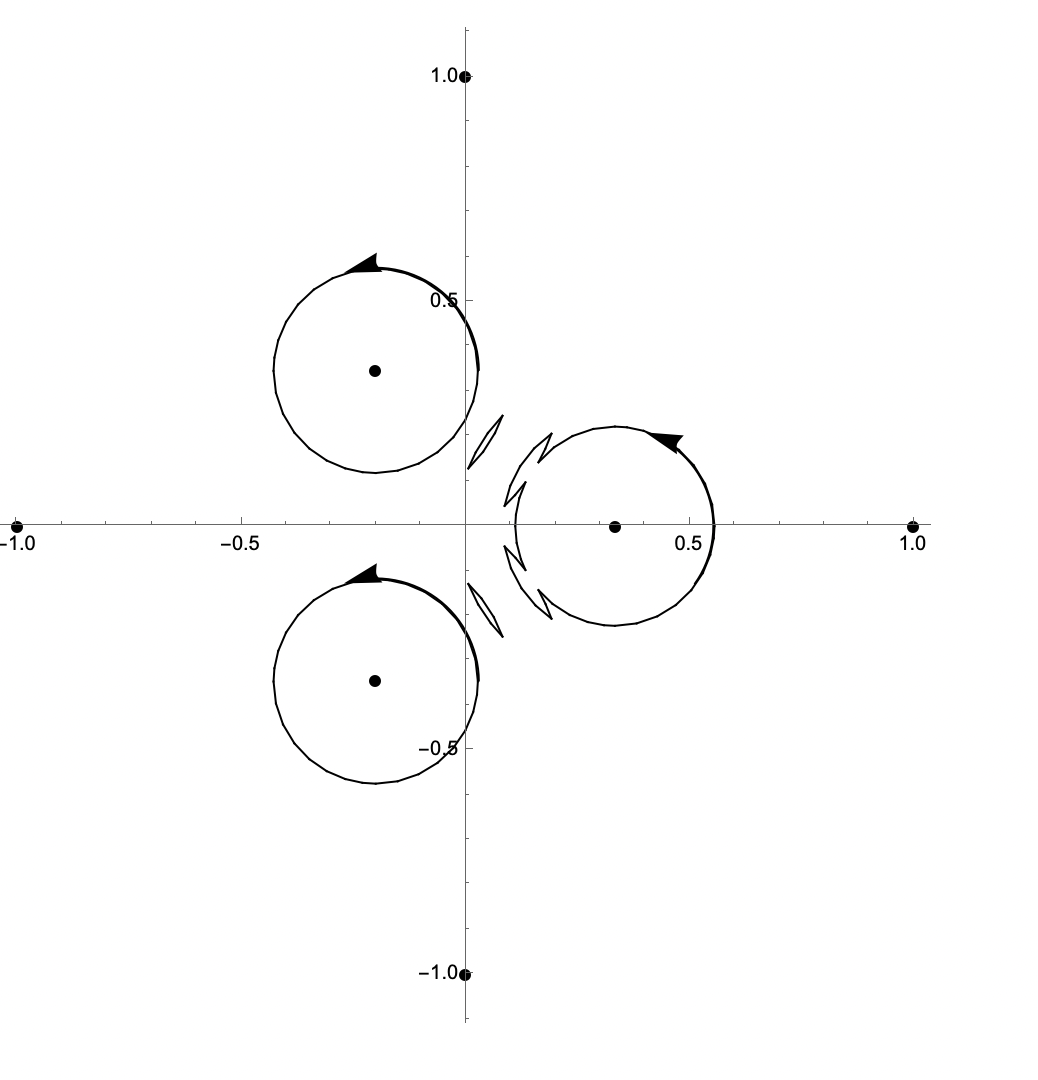}
 \caption{The system of curves $\Gamma_1 = \partial (E_{A} \cap D_A)$ for the example in Figure~\ref{Fig_EADA}, with the zeros $\{\alpha_i\}_{i = 1}^{3}$ of $A$ inside $\Gamma_1$ and the zeros $\{\beta_j\}_{j = 1}^{4}$ of $B$ outside $\Gamma_1$. Recall that $(\alpha_1, \alpha_2, \alpha_3) = (\tfrac{1}{3}, -0.2+0.34641 i, -0.2-0.34641 i)$ and $(\beta_1, \beta_2, \beta_3, \beta_4) = (1, i, -1, -i)$.}
 \label{Fig_Gamma1}
 \end{center}
\end{figure}

Recalling the proof of Theorem~\ref{co:main theorem weak}, it would be natural to estimate the length of $\Gamma_1$ as well as the values of $A$ and $B$ on $\Gamma_1$. However, one can see that the length of $ \Gamma_1 $  depends on the absolute values of the roots $ \alpha_i $ and $ \beta_j $ (which can be very large). The method we will use in Section~\ref{se:main} circumvents this problem by using another system of  contours $ \widetilde{\Gamma}_1 $ defined by
\[
\widetilde{\Gamma}_1:=\Big\{\frac{1}{z}: z\in \Gamma_1\Big\}.
\]
This is well defined, since $ 0\notin\Gamma_1 $. As inversion is a homeomorphism, $ \widetilde{\Gamma}_1 $ is the boundary of 
$$ \Big\{\frac{1}{z}: z\in D_A\cap E_A \Big\} ,$$ which is also a finite union of circular arcs.
Moreover, as we will use this in Section~\ref{se:main}, we may orient $\widetilde{\Gamma}_1$ such that its index with respect to $1/\alpha_i$ is 1 for all $1\le i\le N$, while the index with respect to $1/\beta_j$ is 0 for all $1\le j\le K$.

This next result contains the basic estimates that will be used in the proof of Theorem~\ref{th:main}.

\begin{Lemma}\label{le:estimates on Gamma1}
	With the notation above, the following hold: 
	\begin{itemize}
		\item[(i)] There is a constant $ C_6>0$, depending only on $N$ and $K$, such that  $${\displaystyle \int_{\Gamma_{1}}\frac{|du|}{|u|}\le C_6}.$$

		\item [(ii)] For all $z\in \Gamma_1$, we have that 
		$$|B(z)|\ge \delta/3^{K}\; \;  \mbox{and} \; \;  
		 |A(z)|\ge m:= \min\{C_5\| A \|/2^{N}, \delta /3^{N} \}.$$

	\end{itemize}
\end{Lemma}

\begin{proof}

(i):
For each $1 \leq i \leq N$ let 
\[
\Gamma_{1,i}=\partial F_A^{(i)}.
\]
Clearly we have $ \Gamma_1\subset \bigcup_{i=1}^N \Gamma_{1,i} $.
If $ u\in \Gamma_{1, i} $, then $|u-\alpha_i|\le \tfrac{3}{4}|\alpha_i|$ and therefore,
\begin{equation}\label{eq:|u|>eps on Gamma}
|u|\ge |\alpha_i|-|u-\alpha_i|\ge \tfrac{1}{4}|\alpha_i|.
\end{equation} 
Hence, by~\eqref{eq:length of }, we have 
\[
\int_{\Gamma_{1,i}}\frac{|du|}{|u|}\le 
\frac{ \frac{3\pi N^K}{2}|\alpha_{i}|}{\tfrac{1}{4}|\alpha_i|}
= 6\pi N^K
\]
and thus,
\begin{equation}\label{sdfoooOooooOOO}
\int_{\Gamma_{1}}\frac{|du|}{|u|} \le \sum_{i=1}^N \int_{\Gamma_{1,i}}\frac{|du|}{|u|}
\le  6\pi N^{K+1} =: C_6.
\end{equation}

(ii): If $z\in \Gamma_{1}$, then either $z\in \partial D_A$, in which case~\eqref{estimA1} yields  $|A(z)|\ge C_5 \|A\| /2^{N}$, or $z\in \partial E_A$, in which case $|A(z)|\ge \delta /3^{N}$. 
Since $\Gamma_1\subset \overline{E_A}$,
and $E_A$ is disjoint from $E_B$, we see that $|B|\ge \delta/ 3^{K}$ on~$ \Gamma_1 $.
\end{proof}

Note that the constant $m$, defined in (ii), is not a constant depending only on $N$ and $K$, since it also depends on $\| A \|$.

\section{Proof of the main result}\label{se:main}

As noted above, we have no control of the size of the zeros of $ A $ and $ B $. However, the assumptions made in Section~\ref{se:contours} imply that their inverses are bounded: $ |\alpha_i|, |\beta_j|\ge 2\eps $ implies that 
$$\frac{1}{|\alpha_i|},\frac{1}{|\beta_j|} \le \frac{1}{2\eps}.$$ This is the basis of the proof that follows.

\begin{proof}[Proof of Theorem~\ref{th:main}]
Once again, by Lemma \ref{le:sylvester} and Lemma \ref{le:reduction} we can assume that the roots of $A$ and $B$ are simple. 
Define 
$$\widetilde{A}(z)=
	z^NA(1/z) \;  \; \mbox{and} \; \;  \widetilde{B}(z)=z^KB(1/z).$$ Notice that $\|\widetilde{A}\| = \|A\|$ and $\|\widetilde{B}\| = \|B\|$ since the  coefficients of $\widetilde{A}$ are the reverse of those of $A$ (and similarly for $\widetilde{B}$).
Moreover, the roots of $\widetilde{A}$ are $\widetilde\alpha_i:=1/\alpha_i$ while the roots of $\widetilde{B}$ are $\widetilde\beta_j:=1/\beta_j$. 

%
	Applying Corollary~\ref{co:Bezout extended} when $P(z)=z^{N+K-1}$, one produces polynomials $\widetilde{R}\in\C_{K - 1}[z]$
	and $\widetilde{S}\in\C_{N - 1}[z]$ which satisfy the 
	modified B\'ezout equation 
	\begin{equation}\label{eq:formula with tilde}
		\widetilde A(z)\widetilde{R}(z)+\widetilde B(z)\widetilde{S}(z)=z^{N+K-1},
	\end{equation} 
	and, according to to~\eqref{eq:solutions for arbitrary monomial}, applied to the case $ t=N+K-1 $, they are given by the formulas
\begin{equation}\label{eq:sum form for Sp, Rp - 1}
	\begin{split}
			\widetilde S(z)&=\sum_{j=0}^{N-1}z^j\sum_{k=j+1}^N \widetilde{a}_k\frac{1}{2\pi i}\int_{\widetilde{\Gamma}_1}\frac{\zeta^{N+K+k-j-2}}
		{\widetilde{A}(\zeta)\widetilde{B}(\zeta)}\,d\zeta,\\
			\widetilde R(z)&=\sum_{j=0}^{K-1}z^j\sum_{k=j+1}^K \widetilde{b}_k\frac{1}{2\pi i}\int_{\widetilde{\Gamma}_2}\frac{\zeta^{N+K+k-j-2}}
		{\widetilde{A}(\zeta)\widetilde{B}(\zeta)}\,d\zeta.
	\end{split}
\end{equation}

Let us estimate the integrals above. We have 
\[
\int_{\widetilde{\Gamma}_1}\frac{\zeta^{N+K+k-j-2}}
{\widetilde{A}(\zeta)\widetilde{B}(\zeta)}\,d\zeta=
\int_{\widetilde{\Gamma}_1}\frac{\zeta^{k-j-2}}
{{A}(1/\zeta){B}(1/\zeta)}\,d\zeta
=\int_{\Gamma_1}\frac{1}{u^{k-j}A(u)B(u)}\,du,
\]
where for the last equality we have made the change of variable $u=1/\zeta$.

Applying Lemma~\ref{le:estimates on Gamma1} (ii), it follows that
\[
\Big|\widetilde{a}_k\int_{\widetilde{\Gamma}_{1}}\frac{\zeta^{N+K+k-j-2}}
{\widetilde{A}(\zeta)\widetilde{B}(\zeta)}\,d\zeta\Big|
\le \frac{|\widetilde{a}_k|}{m\delta/3^K}\int_{\Gamma_1}\frac{|du|}{|u|^{k-j}}\le \frac{3^K \|A\|}{m\delta}\int_{\Gamma_1}\frac{|du|}{|u|^{k-j}}.
\]

 Since $\Gamma_1\subset \overline{D_{A}}$,~\eqref{eq:z>eps/2} implies that $|u|\ge \varepsilon/2$ for all $u\in \Gamma_1$. Therefore applying
 Lemma~\ref{le:estimates on Gamma1}  and noting that $k-j-1\ge 0$,  yields
\begin{align*}
\int_{\Gamma_1}\frac{|du|}{|u|^{k-j}}& =\int_{\Gamma_1}\frac{|du|}{|u||u|^{k-j-1}}\le \left(\frac{2}{\varepsilon}\right)^{k-j-1}\int_{\Gamma_1}\frac{|du|}{|u|}\\
&\le C_6\left(\frac{2}{\varepsilon}\right)^{\max(N,K)-1}=:C_7.
\end{align*}
Thus,
\[
\Big|\widetilde{a}_k\int_{\widetilde{\Gamma}_{1}}\frac{\zeta^{N+K+k-j-2}}
{\widetilde{A}(\zeta)\widetilde{B}(\zeta)}\,d\zeta\Big| \le C_7 \frac{3^K \|A\|}{m\delta}
\]
and so
\[
|s_j|\le \frac{NC_7 3^K\| A \|}{m\delta} \; \; \mbox{for all $0 \leq j \leq K-1$.}
\]

Recall that $$m=\min\{C_5\| A \|/2^N, \delta/3^N \}.$$  When $m=C_5\| A \|/2^N$, it follows using \eqref{eq:boundedness of delta} that 
\[
|s_j|\le \frac{NC_73^K 2^N}{C_5\delta}
\le \frac{NC_73^K 2^N T}{C_5\delta^2}.
\]
When $m=\delta/3^N$, taking into account that, by \eqref{eq:after translation}, $\| A \|\le C_4$ , we obtain
\[
|s_j|\le \frac{N C_7 C_4 3^{N+K}}{\delta^2}.
\]
Thus, in  all cases 
\begin{equation}\label{eq:most of the terms}
	 |s_j|\le \frac{C_8}{\delta^2}\; \;  \mbox{for all $1 \leq j \leq  N-1$}
\end{equation} 
 for some constant $C_8>0$.

The estimates are similar for the integral on $\Gamma_2$, and so all coefficients of $ \widetilde{R} $ and $ \widetilde{S} $ are  bounded by $C_{9}\delta^{-2}$, where $C_{9}$ depends only on $N$ and $K$.

  Now observe that \eqref{eq:formula with tilde} can be written as 
\[
 z^NA(1/z)\widetilde{R}(z)+z^KB(1/z)\widetilde{S}(z)=z^{N+K-1}.
\]
Dividing by $z^{N+K-1}$ we obtain
\[
 A(1/z)\widetilde{R}(z)/z^{K-1}+B(1/z)\widetilde{S}(z)/z^{N-1}=1.
\]
Replacing $z$ with $1/z$ yields
\[
 A(z)z^{K-1}\widetilde{R}(1/z)+B(z)z^{N-1}\widetilde{S}(1/z)=1.
\]
Setting $R(z)=z^{K-1}\widetilde{R}(1/z)$ and $S(z)=z^{N-1}\widetilde{S}(1/z)$ gives us
\[
 A(z)R(z)+B(z)S(z)=1.
\]
Thus $R$ and $S$ are the solutions of the desired B\'ezout equation.
Since the coefficients of $R$ and $S$ are those of $\widetilde{R}$ and $\widetilde{S}$ in reverse order, 
it follows that $\|  R \|, \| S \|\le C_{9}\delta^{-2}$. The proof of the theorem, at least when the roots of $A$ and $B$ are simple, is now finished by setting $C=C_{9}$. 

As observed earlier, the assumption that $ A $ and $ B $ have simple roots can be removed by
Lemma~\ref{le:sylvester}.  This completes the proof of Theorem~\ref{th:main}.
\end{proof}

\section{Examples} 
 In this section, we give three examples. The first one shows the optimality of our estimates. The second one discusses the necessity of the assumption that $\|A\|,\|B\|\leq 1$ in Theorem \ref{th:main}. The last one shows that the functions $\delta$ and $\widetilde\delta$ (from \eqref{eq:other delta} and \eqref{eq:tilde delta}) are not continuous.
 
\begin{Example}
For $N \in \N$ let $w=e^{2\pi i/(2N-1)}$ and $0 < a < 1$. A calculation shows that
\[
z^{2N-1}- a^{2N-1}=\prod_{j=1}^{2N-1} (z-a w^j),
\]
from which one derives the B\'{e}zout identity  $AR+BS=1$ with 
$$A(z)=z^N, \quad R(z)=\frac{1}{a^{2N-1}} z^{N-1},$$
$$B(z)=\prod_{j=1}^N  (z-a w^j), \quad \mbox{and} \quad S(z)=-\frac{1}{a^{2N-1}} \prod_{j=N+1}^{2N-1}  (z-a w^j).$$ Here
	$\delta=a^N$ and the unique nonzero coefficient of $R$ has modulus $\delta^{-2+1/N}$.
This shows that the exponent of $ \delta $ in an estimate for $ \|R\| $ and $ \| S \| $ must be at least 	$\delta^{-2+1/N}$. 
\end{Example} 

\subsection*{Open problem} Is 
$$C \delta^{-2+\frac{1}{\max(N,K)}}$$ the best possible estimate for $\| R \|$ and $\| S \|$? 


\begin{Example}
With $A$ and $B$ as in the previous example, define 
$$A_1=a^{-2}A \quad \mbox{and} \quad B_1=a^{-2}B.$$  Observe that  $\| A_1 \|, \| B_1 \|\to\infty$ when $a\to 0$. Then
$R_1=a^2R$, $S_1=a^2S$ (the solutions to B\'{e}zout's identity for $A_1$ and $B_1$), and
$$\delta(A_1, B_1)=a^{-2}\delta(A,B)=a^{N-2}.$$ The unique nonzero coefficient of $R_1$ is 
$$r^{(1)}_{N-1}=a^2a^{1-2N}=a^{-2N+3}.$$
Furthermore, since $0<a<1$, we have $\delta(A_1, B_1)<1$.
Then
\[
r^{(1)}_{N-1}\delta(A_1, B_1)^2= a^{-2N+3}a^{2N-4}= a^{-1}.
\]
The last quantity is not bounded independently of $a$. This shows that in Theorem~\ref{th:main}, without the condition that $\| A \|, \|B  \|\le 1$, there is no constant $C > 0$ such that
\[
\|R\|\le C\delta^{-2}.
\]
\end{Example}

\begin{Example}\label{example-DD}

Let $A(z)=z$, $B(z)=1-z$. Then $\delta(A,B)=1$ and $\widetilde\delta(A,B)= \tfrac{1}{2}$. Let 
$$A_n=z+\frac{1}{n}z^2 \quad \mbox{and} \quad B_n(z)=1-z-\left(\frac{1}{n}+\frac{1}{n^2}\right)z^2.$$
Observe that $\|A_n-A\|=\frac{1}{n}\to 0$ and $\|B_n-B\|=\frac{1}{n}+\frac{1}{n^2}\to 0$ as $n\to\infty$. 
 Clearly the zeros of $A_n$ are $0$ (which is  the zero
of $A$) and $-n$ (which goes to infinity). Then
\[
0 \le \widetilde{\delta}(A_n,B_n)\leq \delta(A_n,B_n) \leq |B_n(-n)|=0,
\]
which implies that $\delta(A_n,B_n)=\widetilde{\delta}(A_n,B_n)=0$. This shows that both $\delta$ and $\widetilde\delta$ are not continuous at $(A, B)$.
\end{Example}

 \section{Final remarks}\label{tennn}
 


\subsection{An extension}\label{sse:extension} In our main theorem, one can relax the assumption that $\|A\|, \|B\| \leq 1$ and prove a similar type of result but with bounds depending on the norm of $A$ and $B$. To see this, take arbitrary polynomials $ A$ and $B$ and set  $M=\max(\|A\|,\|B\|)$. Now apply Theorem~\ref{th:main} to the polynomials 
$ A_1=A/M$ and $B_1=B/M$ to obtain 
$$ \delta(A_1, B_1)= \frac{\delta(A,B)}{M},   \quad R_1=M R,  \quad \mbox{and} \quad S_1=M S.$$Therefore,
  \[
  \|R\|\le \frac{C \max(\| A \|, \| B \|)}{\delta(A,B)^2} \; \; \mbox{and} \; \;   \|S\|\le \frac{C \max(\| A \|, \| B \|)}{\delta(A,B)^2}.
  \]

\subsection{A related problem} The paper \cite{Trent07} considers the related problem of estimating the solutions of the corona problem in $H^\infty$ when the initial data are polynomials. However, our results are not directly comparable to those in \cite{Trent07}. First, the starting problem is not the same since the initial lower bound therein is equivalent to
\[
\delta'(A,B)= \inf_{z\in\D} \{|A(z)|+|B(z)|\}.
\]
Clearly $\widetilde{\delta}(A, B)\le \delta'(A,B)$. However, the example on~\cite[p.~422]{Trent07} shows that these quantities are not equivalent.
Secondly, the resulting solutions $R,S$ to~\eqref{Bezout_C} obtained in~\cite{Trent07} are {\em rational} functions  (not necessarily polynomials) and moreover, the estimates for $\| R \|$ and $\| S \|$ are of order
\[
\left(\frac{1}{\delta'(A,B)}\log\frac{1}{\delta'(A,B)} \right)^{2}.
\]
This is known to be a good, though not necessarily the best, estimate for the solutions of the corona problem in $H^\infty$.

\subsection{Varying the degrees of $A,B$}
As stated from the beginning, the various constants that appear in our results depend only on $N = \operatorname{deg} A$ and $K = \operatorname{deg} B$ (and not the coefficients of $A$ and $B$). If we let $N$ and $K$ increase, one can see that our methods produce constants that increase quite rapidly in $N$ and $K$ (see \eqref{sdfoooOooooOOO} for example). For possible numerical applications, it would be interesting to explore the optimal estimates of these constants as functions of $N$ and $K$.

\subsection{Generalizing to more than two polynomials} B\'ezout's polynomial identity from \eqref{Bezout_C} generalizes to more than two polynomials in that if $ A_1, \dots, A_q\in \C[z] $ have no common zeros, then there exist $ R_1, \dots, R_q\in \C[z] $ with $\deg R_i\le \max\{\deg A_j\}-1$ such that $ A_1R_1+\dots+A_qR_q=1 $ (see, for instance,~\cite[Section 1.3]{Yger}). However, our methods do not extend to provide estimates of $ \| R_i \| $ when $ q\ge 3 $. Note that when $ q\ge 3 $, the solutions $ R_i $ are no longer unique, even if we put restrictions on their degree.
 
\subsection*{Open problem} Suppose  that $ A_1, \dots, A_q\in \C[z] $ have no common zeros, and define 
\[
\delta=\min \Big\{ \sum_{i=1}^q |A_i(z)| : z\in\C, \prod_{i = 1}^q A_i(z)=0  \Big\}.
\]
Prove there exists a constant $ C > 0 $, depending only on $ \max\{\deg A_i\}$, such that one can find $ R_1, \dots, R_q\in\C[z] $ with $ A_1R_1+\dots+A_qR_q=1 $ and 
\[
\max_{0\le i\le q}\| R_i \|\le \frac{C}{\delta^2}.
\]

\section*{Acknowledgements}

We are indebted to Laurent Baratchart for pointing out the problem of estimating the norm of the inverse of the Sylvester matrix.

Emmanuel Fricain was supported the Labex CEMPI (ANR-11-LABX -0007-01). Andreas Hartmann was supported by the Project REPKA (ANR-18-CE40-0035). Dan Timotin was partially supported by a grant of the Ministry of Research, Innovation and Digitization, CNCS/CCCDI – UEFISCDI, Project Number PN-III-P4-ID-PCE-2020-0458, within PNCDI III,  and by the international research network ECO-Math.

\bibliographystyle{plain}

\bibliography{references}

\end{document}